\documentclass[12pt]{article}
\usepackage{fullpage,amsthm,amssymb,amsmath}
\usepackage{enumerate,color}
\usepackage{tikz}

\newtheorem{thm}{Theorem}[section]

\newtheorem{lem}[thm]{Lemma}
\newtheorem{cor}[thm]{Corollary}
\newtheorem{conj}[thm]{Conjecture}
\newtheorem{ques}[thm]{Question}
\newcommand\ZZ{\mathbb{Z}}
\newcommand\RR{\mathbb{R}}
\DeclareMathOperator{\lcm}{lcm}
\newcommand\needed[1]{
}
\newcommand\sout{}

\begin{document}

\title{On tiling the integers with $4$-sets of the same gap sequence}

\author{
Ilkyoo Choi\thanks{Corresponding author.
Supported by the National Research Foundation of Korea (NRF) grant funded by the Korea government (MSIP) (NRF-2015R1C1A1A02036398).
Department of Mathematical Sciences, KAIST, Daejeon, Republic of Korea.
\texttt{ilkyoo@kaist.ac.kr}
}
\and
Junehyuk Jung\thanks{Department of Mathematical Sciences, KAIST, Daejeon, Republic of Korea.
\texttt{junehyuk@kaist.ac.kr}
}
\and
Minki Kim\thanks{Department of Mathematical Sciences, KAIST, Daejeon, Republic of Korea.
\texttt{kmk90@kaist.ac.kr}
}
}

\date\today

\maketitle

\begin{abstract}
Partitioning a set into similar, if not, identical, parts is a fundamental research topic in combinatorics.
The question of partitioning the integers in various ways has been considered throughout history.
Given a set $\{x_1, \ldots, x_n\}$ of integers where $x_1<\cdots<x_n$, let the {\it gap sequence} of this set be the nondecreasing sequence $d_1, \ldots, d_{n-1}$ where $\{d_1, \ldots, d_{n-1}\}$ equals $\{x_{i+1}-x_i:i\in\{1,\ldots, n-1\}\}$ as a multiset.
This paper addresses the following question, which was explicitly asked by Nakamigawa: can the set of integers be partitioned into sets with the same gap sequence?
The question is known to be true for any set where the gap sequence has length at most two.
This paper provides evidence that the question is true when the gap sequence has length three.
Namely, we prove that given positive integers $p$ and $q$, there is a positive integer $r_0$ such that for all $r\geq r_0$, the set of integers can be partitioned into $4$-sets with gap sequence $p, q$, $r$.
\end{abstract}

\maketitle

\section{Introduction}

Let $[n]$ denote the set $\{1, \ldots, n\}$ and let $[a, b]$ denote the set $\{a, \ldots, b\}$.
Note that $[1, 0]=\emptyset$.
An $n$-set is a set of size $n$.

Partitioning a set into similar, if not, identical, parts is a fundamental research topic in combinatorics.
In the literature, it is typically said that $T$ {\it tiles} $S$ if the set $S$ can be partitioned into parts that are all ``similar'' to $T$ in some sense.
For example, Golomb initiated the study of tilings of the checker board with ``polyominoes'' in 1954~\cite{1954Go}, and it has attracted a vast audience of both mathematicians and non-mathematicians.
% to be on the Mathematics Subject Classification used by Math Reviews.
See the book by Golumb~\cite{1994Go} for recent developments regarding this particular problem.

The question of partitioning the integers $\ZZ$ (and the positive integers $\ZZ^+$) in various ways has been considered throughout history.
For two sets $T$ and $S$ where $T\subseteq S$ and a group $G$ acting on $S$, we say that ``$T$ tiles $S$ under $G$" if $S$ can be partitioned into copies that are obtainable from $T$ via $G$; namely, there is a subset $X$ of $G$ such that $S=\amalg_{\gamma\in X}\gamma(T)$.
%For two sets $T$ and $S$ and some group action $\gamma$, we say that ``$T$ tiles $S$ under $\gamma$" if $S$ can be partitioned into copies that are obtainable from $T$ via $\gamma$; namely, there is a set $X$ where $S=\bigcup_{x\in X}\gamma(T,x)$ and $\gamma(T,y)\cap \gamma(T,z) \neq\emptyset$ if and only if $x=y$.
Tilings of $\ZZ$ and $\ZZ^+$ under translation have already been extensively studied~\cite{1950Br,1967Lo}.
It is known that a set $S$ of integers tiles $\ZZ^+$ under translation if and only if $S$ tiles some interval of $\ZZ$ under translation.
In particular, a $3$-set $S$ tiles $\ZZ^+$ under translation if and only if the elements of $S$ form an arithmetic progression.

It is easy to see that an arbitrary $2$-set of integers tiles an interval of $\ZZ$ (and therefore tiles $\ZZ$) under translation, and there are $3$-sets of integers that do not tile $\ZZ$ under translation.
However, if both translation and reflection are allowed, then Sands and Swierczkowski~\cite{1960SaSw} provided a short proof that an arbitrary $3$-set of real numbers tiles $\RR$ (simplifying a proof in~\cite{1958KoSe}), and on the way they also proved that an arbitrary $3$-set of integers tiles $\ZZ$.
It is also known that not all $4$-sets of integers tile $\ZZ$ under translation and reflection.

In his book~\cite{1976Ho}, Honsberger strengthened the previous result with a simple greedy algorithm by showing that an arbitrary $3$-set of integers tiles an interval of $\ZZ$ under translation and reflection.
Meyerowitz~\cite{1988Me} analyzed this algorithm and gave a constructive proof that the algorithm produces a tiling of an interval of $\ZZ$, and also proved that a $3$-set of real numbers tiles $\RR^+$, strengthening an aforementioned result.
This algorithm does not necessarily find the shortest interval of $\ZZ$ that a $3$-set of integers can tile;
there has been effort in trying to determine the shortest such interval~\cite{1981AlHo,2000Na}, and in some cases the shortest such interval is known.

%A set $S$ of size $3$ tiles $\ZZ^+$ under the full euclidean group.
Gordan~\cite{1980Go} generalized the problem to higher dimensions.
He proved that a $3$-set of $\ZZ^n$ tiles $\ZZ^n$ under the Euclidean group actions (translation, reflection, and rotation), and that there is a set of size $4n-2\lfloor{n/2}\rfloor$ of $\ZZ^n$ that does not tile $\ZZ^n$ under the Euclidean group actions.
More information regarding higher dimensions is in Section~\ref{sec:open}.
There is also a paper~\cite{2015Na} that studies tilings of the cyclic group $\ZZ_n$.

This paper focuses on partitioning $\ZZ$ into sets with the same ``gap sequence'' and ``gap length'', which is the term used in~\cite{2015Na} and~\cite{2005Na}, respectively.
Given a set $\{x_1, \ldots, x_n\}$ of integers where $x_1<\cdots<x_n$, let the {\it gap sequence} of this set be the nondecreasing sequence $d_1, \ldots, d_{n-1}$ where $\{d_1, \ldots, d_{n-1}\}$ equals $\{x_{i+1}-x_i:i\in\{1,\ldots, n-1\}\}$ as a multiset.
Note that the gap sequence of a set with $n$ elements has length $n-1$.
Roughly speaking, in addition to reflecting the order of the gaps of a given set, any permutation of the order of the gaps of the set is allowed.

%A set $\{x_1, \ldots, x_n\}$ is a $(d_1, \ldots, d_{n-1})$-{\it set} if $x_1<\cdots<x_n$ and $\{x_{i+1}-x_{i}:i\in[n-1]\}$ equals $\{d_1, \ldots, d_{n-1}\}$ as a multiset.
%Note that a $(d_1, \ldots, d_{n-1})$-set  has size $n$.
%It is worth mentioning that a $(d_1, \ldots, d_m)$-set is also a $(d_{\sigma(1)}, \ldots, d_{\sigma(m)})$-set for any permutation $\sigma$ of $[m]$;
%in other words, the order of $d_1, \ldots, d_m$ does not matter.

In~\cite{2005Na}, the following question was explicitly asked:

\begin{ques}[\cite{2005Na}]\label{ques}
For a gap sequence $S$ of length $n-1$, can $\ZZ$ be partitioned into $n$-sets with the same gap sequence $S$?
\end{ques}

Since allowing permutations of the order of the gaps of a given set does not provide additional help (when reflections of the gaps are already allowed), previous results above imply that this question is true when $n\in\{1, 2, 3\}$.
%For any positive integer $p$ it is easy to see that, $(p)$-sets tile the set of integers, and for any pair of positive integers $p$ and $q$, it is known that $(p, q)$-sets also tile the set of integers.
%For any three positive integers $p, q$, and $r$, it is conjectured that $(p, q, r)$-sets tile the set of integers.
%
%It is easy to see that for any positive integer $p$, the set of integers can be partitioned into sets (of two integers) so that each set is of $(p)$-set, and it is known that for any pair of positive integers $p$ and $q$, the set of integers can be partitioned into sets (of three integers) so that each set is of $(p, q)$-set.
%It is conjectured that for any three positive integers $p, q$, and $r$, the set of integers can be partitioned into sets (of four integers) so that each set is of $(p, q, r)$-set.
In this paper, we prove the following theorem that provides evidence that the question is true when $n=4$.
Corollary~\ref{cor:easy} is an immediate consequence of the theorem.

\begin{thm}\label{thm:main}
%\mainstatement
There is an interval of the integers that can be partitioned into $4$-sets with the same gap sequence $p, q, r$, if $q\geq p$ and $r\geq \max\{4q(4q-1),  {\frac{1}{\gcd(p, q)}}({5p+4q}-\gcd(p,q))({4p+3q}-\gcd(p,q)) \}$.
\end{thm}

\begin{cor}\label{cor:easy}
There is an interval of the integers that can be partitioned into $4$-sets with the same gap sequence $p, q, r$, if $r\geq 63(\max\{p, q\})^2$.
\end{cor}

Note that for the sake of presentation, we omit some improvements on the constants of the threshold on $r$.
Our proof follows the ideas in~\cite{2000Na,2005Na}, where partitions of $\ZZ^2$ is used to aid the partition of $\ZZ$.
We develop and push the method further and generalize it to $\ZZ^3$.
In Section~\ref{sec:lemmas}, we show that we can partition certain subsets of $\ZZ^3$ into smaller subsets of $\ZZ^3$ that we call {\it blocks}.
In Section~\ref{sec:main}, we demonstrate how to use the lemmas in Section~\ref{sec:lemmas} to tile an interval of $\ZZ$ with $4$-sets with the desired gap sequence.
We finish the paper with some open questions in Section~\ref{sec:open}.

\section{Lemmas}\label{sec:lemmas}

Given three vectors $d_1, d_2, d_3$ in $\ZZ^3$, a $4$-set $\{v_1, v_2, v_3, v_4\}$ of $\ZZ^3$ in which $\{v_4-v_3, v_3-v_2, v_2-v_1\}=\{d_1, d_2, d_3\}$ is called a {\it $(d_1, d_2, d_3)$-block}.
For a set $V$ of triples of vectors in $\ZZ^3$, we say a set $S$ of $\ZZ^2$ can be {\it covered} (with {\it height $h(S)$}) by $V$-blocks if there exists an integer $h(S)$ such that $S\times h(S)=\{(x, y, z):(x, y)\in S, z\in[h(S)]\}$ can be partitioned into blocks from $V$.
If $V$ only has one vector $v$, then we simply write ``covered by $v$-blocks'' instead of ``covered by $\{v\}$-blocks''.

Let $e_1=(1, 0, 0)$, $e_2=(0, 1, 0)$, and $e_3=(0, 0, 1)$ be unit vectors in $\ZZ^3$.
By stretching $X\subset \ZZ^3$ in the $e_1$, $e_2$, and $e_3$ direction by a real number $w$, we obtain $\{(wx, y, z): (x, y, z)\in X\}$, $\{(x, wy, z): (x, y, z)\in X\}$, and $\{(x, y, wz): (x, y, z)\in X\}$, respectively.

\subsection{When $q\geq 2p$}\label{subsec:qgeq2p}

\begin{lem}\label{lem:block1}
The following sets of $\ZZ^2$ can be covered by $(e_1, e_2, e_3)$-blocks:
\begin{enumerate}[$(i)$]
\item $S_1=\{(1, 1), (1, 2), (2, 2)\}$ with $h(S_1)=4$
\item $S_2=\{(1, 1), (2, 1), (2, 2)\}$ with $h(S_2)=4$
\item $S_3=[3]\times [2]$ with $h(S_3)=4$
\item $S_4=[k]\times[4]$ for $k\geq 2$ with $h(S_4)=20$
\item $S_5=([2]\times [4])\cup\{(3, 1),(3,2)\}$ with $h(S_5)=4$
\item $S_6=([2]\times [4])\cup\{(3, 4)\}$ with $h(S_6)=4$
\item $S_7=([ k]\times[4])\cup\{(k+1, 4)\}$ for $k\geq 2$ with $h(S_7)=20$
\needed{
\item $S_{10}=[4]\times [2]$\hfill NECESSARY?
\item $S_{11}=[5]\times [2]$\hfill NECESSARY?
\item $S_{12}=[3]\times [3]$\hfill NECESSARY?}
\end{enumerate}
\end{lem}
\sout{
\begin{proof}
See Figure~\ref{fig:block1} for an illustration of some cases.

$(i)$:
\begin{center}
$B_1=\{(1, 1, 1), (1, 2, 1), (2, 2, 1), (2, 2, 2)\}$,

$B_2=\{(1, 1, 2), (1, 2, 2), (1, 2, 3), (2, 2, 3)\}$,

$B_3=\{(1, 1, 3), (1, 1, 4), (1, 2, 4), (2, 2, 4)\}$.
\end{center}

$(ii)$:
\begin{center}
$B_1=\{(1, 1, 1), (2, 1, 1), (2, 2, 1), (2, 2, 2)\}$,

$B_2=\{(1, 1, 2), (2, 1, 2), (2, 1, 3), (2, 2, 3)\}$,

$B_3=\{(1, 1, 3), (1, 1, 4), (2, 1, 4), (2, 2, 4)\}$.
\end{center}

$(iii)$:
Combine $S_1$ and the block obtained by shifting $S_2$ by $e_1$.

$(iv)$:
If $k$ is even, then we can do better and obtain $h(S_4)=5$.
It is not hard to see we can fill $S_4$ with blocks of $[ 2]\times [ 4]$ by putting them side by side, so it is sufficient to show how to fill $[ 2]\times [ 4]$.

\begin{center}
$B_1=\{(1, 1, 1), (2, 1, 1), (2, 1, 2), (2, 2, 2)\}$,

$B_2=\{(1, 2, 1), (2, 2, 1), (2, 3, 1), (2, 3, 2)\}$,

$B_3=\{(1, 3, 1), (1, 4, 1), (2, 4, 1), (2, 4, 2)\}$,

$B_4=\{(1, 1, 2), (1, 2, 2), (1, 2, 3), (2, 2, 3)\}$,

$B_5=\{(1, 3, 2), (1, 4, 2), (1, 4, 3), (2, 4, 3)\}$,

$B_6=\{(1, 1, 3), (2, 1, 3), (2, 1, 4), (2, 2, 4)\}$,

$B_7=\{(1, 3, 3), (2, 3, 3), (2, 3, 4), (2, 4, 4)\}$,

$B_8=\{(1, 1, 4), (1, 1, 5), (2, 1, 5), (2, 2, 5)\}$,

$B_9=\{(1, 2, 4), (1, 2, 5), (1, 3, 5), (2, 3, 5)\}$,

$B_{10}=\{(1, 3, 4), (1, 4, 4), (1, 4, 5), (2, 4, 5)\}$.
\end{center}

If $k$ is odd, then $h(S_4)=20$.
It is not hard to see we can fill $S_4$ with two copies of $S_3$ (which was already shown to be covered in $(iii)$) by putting one on top of another and then using blocks of $[2]\times [4]$ side by side.
Note that the least common multiple of $h(S_3)=4$ and $h([2]\times[4])=5$ is $20$.

$(v)$:
\begin{center}
$B_1=\{(1, 1, 1), (1, 1, 2), (1, 2, 2), (2, 2, 2)\}$,

$B_2=\{(1, 2, 1), (2, 2, 1), (2, 3, 1), (2, 3, 2)\}$,

$B_3=\{(2, 1, 1), (3, 1, 1), (3, 2, 1), (3, 2, 2)\}$,

$B_4=\{(1, 3, 1), (1, 4, 1), (2, 4, 1), (2, 4, 2)\}$,

$B_5=\{(2, 1, 2), (3, 1, 2), (3, 1, 3), (3, 2, 3)\}$,

$B_6=\{(1, 3, 2), (1, 4, 2), (1, 4, 3), (2, 4, 3)\}$,

$B_7=\{(1, 1, 3), (1, 1, 4), (1, 2, 4), (2, 2, 4)\}$,

$B_8=\{(1, 2, 3), (2, 2, 3), (2, 3, 3), (2, 3, 4)\}$,

$B_9=\{(2, 1, 3), (2, 1, 4), (3, 1, 4), (3, 2, 4)\}$,

$B_{10}=\{(1, 3, 3), (1, 3, 4), (1, 4, 4), (2, 4, 4)\}$.
\end{center}

$(vi)$:
\begin{center}
$B_1=\{(1, 1, 1), (2, 1, 1), (2, 2, 1), (2, 2, 2)\}$,

$B_2=\{(1, 2, 1), (1, 2, 2), (1, 3, 2), (2, 3, 2)\}$,

$B_3=\{(1, 3, 1), (1, 4, 1), (1, 4, 2), (2, 4, 2)\}$,

$B_4=\{(2, 3, 1), (2, 4, 1), (3, 4, 1), (3, 4, 2)\}$,

$B_5=\{(1, 1, 2), (2, 1, 2), (2, 1, 3), (2, 2, 3)\}$,

$B_6=\{(1, 1, 3), (1, 1, 4), (2, 1, 4), (2, 2, 4)\}$,

$B_7=\{(1, 2, 3), (1, 2, 4), (1, 3, 4), (2, 3, 4)\}$,

$B_8=\{(1, 3, 3), (1, 4, 3), (1, 4, 4), (2, 4, 4)\}$,

$B_9=\{(2, 3, 3), (2, 4, 3), (3, 4, 3), (3, 4, 4)\}$.
\end{center}

$(vii)$:
Assume $k$ is even.
It is not hard to see we can fill $S_7$ with one $S_6$ (which was already shown to be covered in $(vi)$) and then using blocks of $[ 2]\times [ 4]$ side by side.
Note that the least common multiple of $h(S_6)=4$ and $h([2]\times[4])=5$ is $20$.

Assume $k$ is odd.
When $k=3$, it is not hard to see that we can fill $S_7$ with one $S_1$ (which was already shown to be covered in $(i)$) and one $S_5$ (which was already shown to be covered in $(v)$).
When $k>3$, attach blocks of $[2]\times [4]$ side by side to the configuration when $k=3$.
Note that the least common multiple of $h(S_6)=4$ and $h([2]\times[4])=5$ is $20$.
%\begin{center}
%Combine the block in $(vi)$ and the block obtained by shifting the block in $(ii)$ by $2e_1+2e_2$.
%\end{center}
%%%%%%%%%%%% NECESSARY?
\needed{
\\$(x)$:
$h(S_{10})=5$.
NECESSARY?\\
$(xi)$:
$h(S_{11})=8$.
NECESSARY?\\
$(xii)$:
$h(S_{12})=4$.
NECESSARY?}
%%%%%%%%%%%%%
\end{proof}
}

\begin{figure}[h]
	\begin{center}
  \includegraphics[scale=0.5]{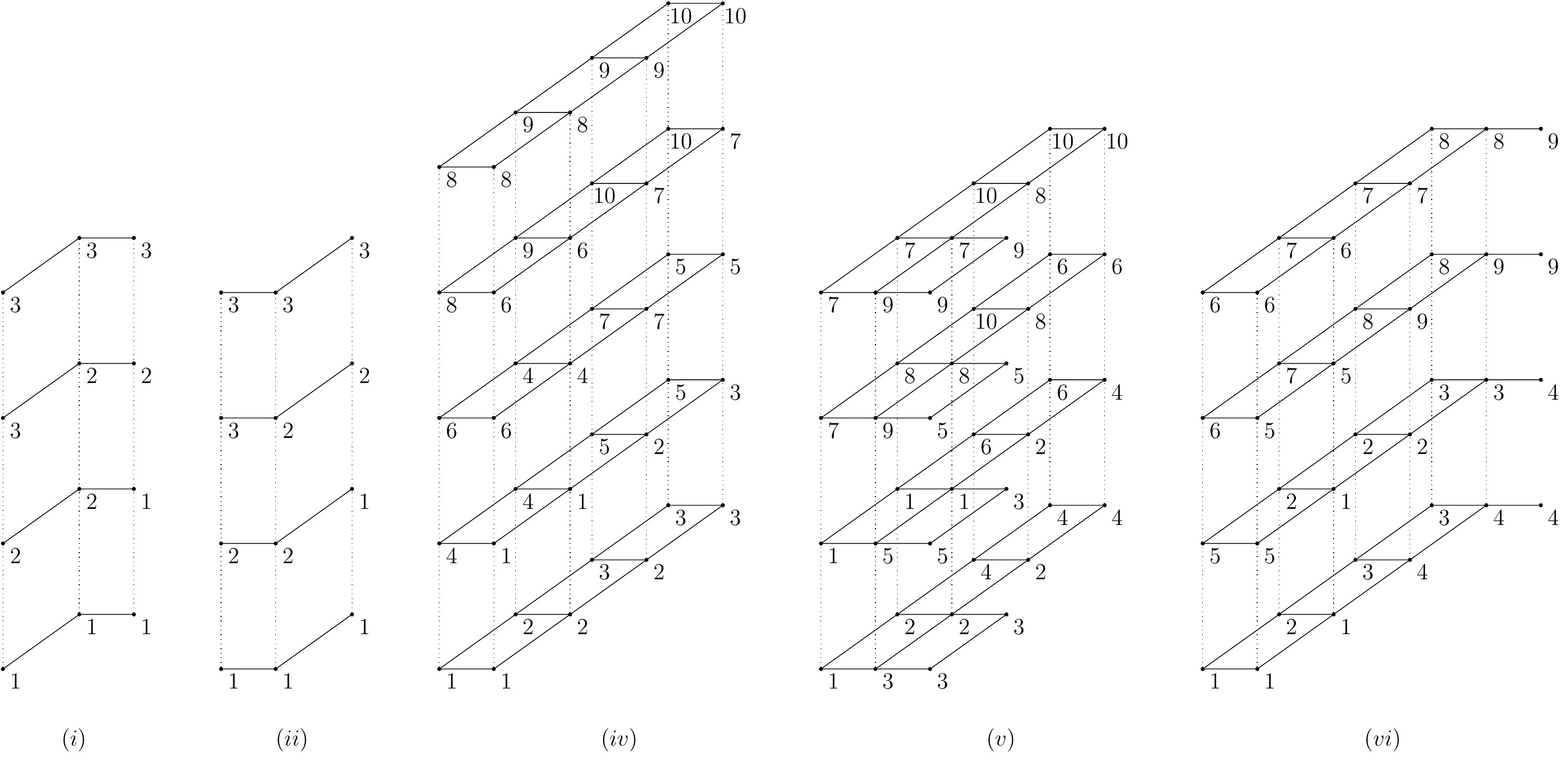}
  \caption{Illustration for some cases of Lemma~\ref{lem:block1}.}
  \label{fig:block1}
	\end{center}
\end{figure}

\begin{lem}\label{lem:layer1}
Given $q\geq 2p$, the following sets can be covered by $(pe_1, e_2, e_3)$-blocks:
\begin{enumerate}[(i)]
\item $X_1=[ q]\times[4]$ with $h(X_1)=20$
\item $X_2=([ q]\times[4])\cup\{(q+1, 4)\}$ with $h(X_2)=20$
\end{enumerate}
\end{lem}
\sout{
\begin{proof}
%Since $(pe_1, e_2, e_3)$ is exactly the same as stretching $(e_1, e_2, e_3)$ in the $e_1$ direction by $p$,
It is sufficient to show that $X_1$ and $X_2$ can be covered by $(e_1, e_2, e_3)$-blocks, but stretched by $p$.

Let $a=\lfloor{\frac{q}{p}}\rfloor$ and $b=q-ap$ so that $b\in[0, p-1]$.
Note that $a\geq 2$ since $q\geq 2p$.
Obtain $P^{a+1}_4$, $P^{a}_4$, and $P^{a}_7$ by stretching $S_4$ with $k=a+1$, $S_4$ with $k=a$, and $S_7$ with $k=a$, respectively, from Lemma~\ref{lem:block1} in the $e_1$ direction by $p$;
in other words,
$P^{a+1}_4=\{(1+ip, y):i\in[0, a], y\in[4]\}$,
$P^{a     }_4=\{(1+ip, y):i\in[0, a-1], y\in[4]\}$, and
$P^{a}_7=P^{a}_4\cup\{(1+ap, 4)\}$.

Let $P^*_4=\{P^{a+1}_4+(i,0): i\in[0,  b-1]\}$ and $P^{**}_4=\{P^a_4+(i,0):i\in[b+1, p-1]\}$.

$(i)$:
Now $P^*_4$, $P^a_4+(b,0)$, $P^{**}_4$ is a partition of $X_1$.
Since $S_4$ can be covered with height $20$, we conclude that $X_1$ can be covered with height $20$.

$(ii)$:
Now $P^*_4$, $P^a_7+(b,0)$, $P^{**}_4$ is a partition of $X_2$.
Since $S_4$ and $S_7$ can be covered with height $20$, we conclude that $X_2$ can be covered with height $20$.

See Figure~\ref{fig:layer1} for an illustration.
\end{proof}
}

\begin{figure}[h]
	\begin{center}
  \includegraphics[scale=1]{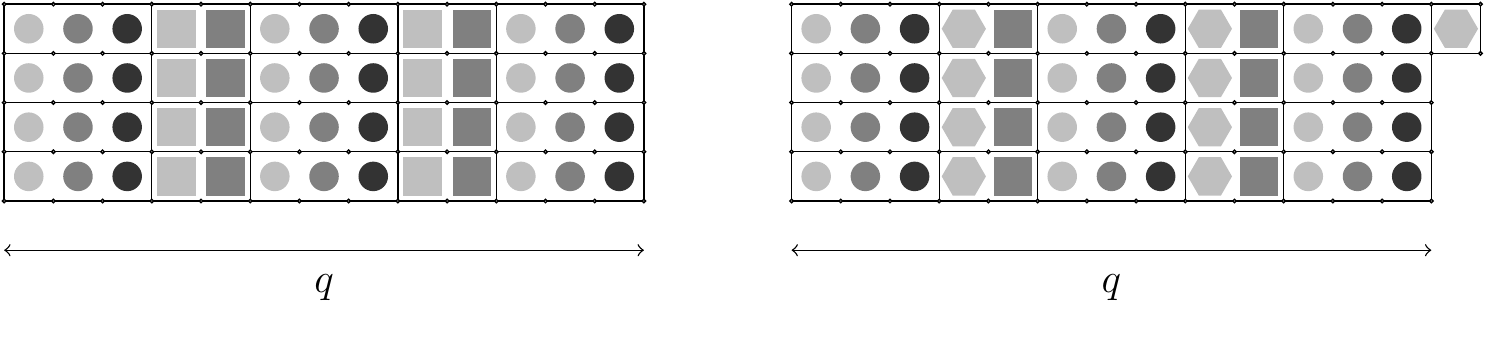}
  \caption{Figure for Lemma~\ref{lem:layer1}. Same shape and same shade means same block. Same shape and different shade means same type of block, but different block.}
  \label{fig:layer1}
	\end{center}
\end{figure}

\subsection{When $p\leq q\leq 2p$}\label{subsec:qleq2p}

\begin{lem}\label{lem:block2}
The following sets of $\ZZ^2$ can be covered by $(e_1, e_2-e_1, e_3)$-blocks:
\begin{enumerate}[$(i)$]
\item $T_1=\{(1, 1), (1, 2), (2, 1)\}$   with $h(T_1)=4$
\item $T_2=\{(1, 2), (2, 1), (2, 2)\}$   with $h(T_2)=4$
\item $T_3=\{(1, 2), (1,3), (2, 1), (2, 2)\}$   with $h(T_3)=2$
\item $T_4=[1,2]\times [1,2]$   with $h(T_4)=2$
\item $T_5=\{(1,1), (1,2), (1,3), (2,1), (2,2), (3,1)\}$ with $h(T_5)=2$% 추가
\needed{ \item $T_5=[1,3]\times [1,3]$\hfill NECESSARY?}
\end{enumerate}
\end{lem}
\sout{
\begin{proof}
See Figure~\ref{fig:block2} for an illustration.

$(i)$:
\begin{center}
$B_1=\{(1, 1, 1), (2, 1, 1), (1, 2, 1), (1, 2, 2)\}$,

$B_2=\{(1, 1, 2), (2, 1, 2), (2, 1, 3), (1, 2, 3)\}$,

$B_3=\{(1, 1, 3), (1, 1, 4), (2, 1, 4), (1, 2, 4)\}$.
\end{center}

$(ii)$:
\begin{center}
$B_1=\{(2, 1, 1), (1, 2, 1), (2, 2, 1), (2, 2, 2)\}$,

$B_2=\{(2, 1, 2), (1, 2, 2), (1, 2, 3), (2, 2, 3)\}$,

$B_3=\{(2, 1, 3), (2, 1, 4), (1, 2, 4), (2, 2, 4)\}$.
\end{center}

$(iii)$:
\begin{center}
$B_1=\{(2, 1, 1), (2, 1, 2), (1, 2, 2), (2, 2, 2)\}$,

$B_2=\{(1, 2, 1), (2, 2, 1), (1, 3, 1), (1, 3, 2)\}$.
\end{center}

$(iv)$:
\begin{center}
$B_1=\{(1, 1, 1), (1, 1, 2), (2, 1, 2), (1, 2, 2)\}$,

$B_2=\{(2, 1, 1), (1, 2, 1), (2, 2, 1), (2, 2, 2)\}$.
\end{center}

$(v)$:
\begin{center}
$B_1=\{(1, 1, 1), (1, 1, 2), (2, 1, 2), (1, 2, 2)\}$,

$B_2=\{(2, 1, 1), (3, 1, 1), (3, 1, 2), (2, 2, 2)\}$,

$B_3=\{(1,2,1), (2,2,1), (1,3,1), (1,3,2)\}$.
\end{center} % 추가 항목

%%%%%%%%%%%%%
\needed{
$(v)$:
$h(T_5)=4$.
NECESSARY?}
%%%%%%%%%%%%%%%%%
\end{proof}
}

\begin{figure}[h]
	\begin{center}
  \includegraphics[scale=0.85]{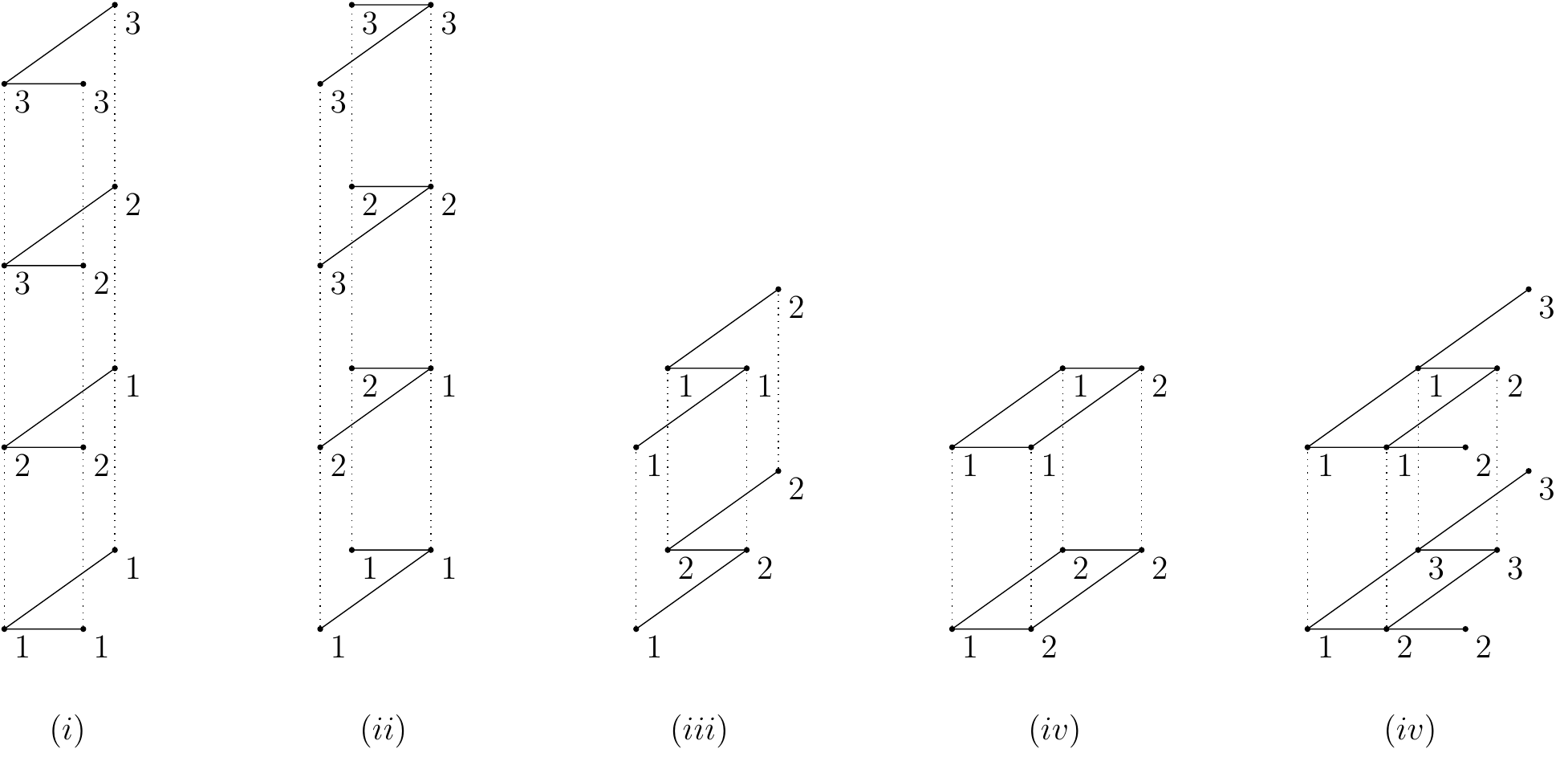}
  \caption{Illustration for Lemma~\ref{lem:block2}.}
  \label{fig:block2}
	\end{center}
\end{figure}

% 다음과 같은 그림을 추가
% \begin{figure}[htbp]
% \tikzstyle{v}=[circle, draw, solid, fill=black, inner sep=0pt, minimum width=1pt]
%  \centering
%  \begin{tikzpicture}[scale=1]
%    \node[v](v1) at (0,0,0){}; \node[right] at (-0.05,-0.2,0){1};
%    \node[v](v2) at (0,2,0){}; \node[right] at (-0.05, 1.8, 0){1};
%    \node[v](v3) at (0.9,2,0){}; \node[right] at (0.85, 1.8, 0){1};
%    \node[v](v4) at (0,2,-3.2){}; \node[right] at (-0.05, 1.8,-3.2){1};
%    \node[v](v5) at (0.9,0,0){}; \node[right] at (0.85, -0.2, 0){2};
%    \node[v](v6) at (1.8,0,0){}; \node[right] at (1.75, -0.2, 0){2};
%    \node[v](v7) at (1.8,2,0){}; \node[right] at (1.75, 1.8, 0){2};
%    \node[v](v8) at (0.9,2,-3.2){}; \node[right] at (0.85, 1.8, -3.2){2};
%    \node[v](v9) at (0,0,-3.2){}; \node[right] at (-0.05,-0.2,-3.2){3};
%    \node[v](v10) at (0.9,0,-3.2){}; \node[right] at (0.85, -0.2, -3.2){3};
%    \node[v](v11) at (0,0,-6.4){}; \node[right] at (-0.05,-0.2,-6.4){3};
%    \node[v](v12) at (0,2,-6.4){}; \node[right] at (-0.05,1.8,-6.4){3};
%    \draw (v1)--(v5)--(v6); \draw (v9)--(v10); \draw (v1)--(v9)--(v11); \draw (v5)--(v10);
%    \draw (v2)--(v3)--(v7); \draw (v4)--(v8); \draw (v2)--(v4)--(v12); \draw (v3)--(v8);
%    \draw[dotted] (v1)--(v2); \draw[dotted] (v3)--(v5); \draw[dotted] (v6)--(v7);
%    \draw[dotted] (v4)--(v9); \draw[dotted] (v11)--(v12); \draw[dotted] (v8)--(v10);
%   \end{tikzpicture}
%  \caption{v5}
%  \label{fig:v5}
%\end{figure}

\begin{lem}\label{lem:layer2}
Given $p\leq q\leq 2p$, the following sets can be covered by $\{(pe_1, e_2-pe_1, e_3), (qe_1, e_2-qe_1, e_3)\}$-blocks:
\begin{enumerate}[(i)]
\item $Y_1=([p+q]\times[4])\cup \{(x,5): x\in[p]\}$  with $h(Y_1)=4$ % 2p+2q를 5p+4q로 수정
\item $Y_2=([p+q]\times[3])\cup \{(x, 4): x\in[p]\}$  with $h(Y_2)=4$
\end{enumerate}
\end{lem}
\sout{
\begin{proof}
%Since $(pe_1, e_2-pe_1, e_3)$ and $(qe_1, e_2-qe_1, e_3)$ is exactly the same as stretching $(e_1, e_2-e_1, e_3)$ in the $e_1$ direction by $p$ and $q$, respectively,
It is sufficient to show that $Y_1$ and $Y_2$ can be covered by $(e_1, e_2-e_1, e_3)$-blocks, but stretched by either $p$ or $q$.

Let $q=p+t$ so that $t\in[0, p]$.
% 수정
Obtain $P_1$, $P_2$, $P_3$, $P_4$, and $P_5$ by stretching $T_1$, $T_2$, $T_3$, $T_4$, and $T_5$ respectively, from Lemma~\ref{lem:block2} in the $e_1$ direction by $p$; in other words $P_1=\{(1, 1), (p+1, 1), (1, 2)\}$, $P_2=\{(1, 2), (p+1, 1), (p+1, 2)\}$, $P_3=\{(1, 2), (1, 3), (1+p, 1), (1+p, 2)\}$, $P_4=\{(1, 1), (1, 2), (1+p, 1), (1+p, 2)\}$,
$P_5=\{(1,1), (1,2), (1,3), (1+p,1), (1+p,2), (1+2p,1)\}$.
% 수정
Obtain $Q_1$ by stretching $T_1$ from Lemma~\ref{lem:block2} in the $e_1$ direction by $q$; in other words $Q_1=\{(1, 1), (q+1, 1), (1, 2)\}$.

% 수정
$(i)$:
Let $P^*_1=\{P_1+(i,0):i\in[t,p-1]\}$, $P^*_2=\{P_2+(i,1):i\in[t,p-1]\}$,
$P^*_3=\{P_3+(i,1):i\in[p,p+t-1]\}$, $P^*_5=\{P_5+(i,0):i\in[0,t-1]\}$,
 and $Q^*_1=\{Q_1+(i,3):i\in[0,p-1]\}$.
Now $P^*_1, P^*_2, P^*_3, P^*_5, Q^*_1$ is a partition of $Y_1$.

$(ii)$:
Let $P^{**}_1=\{P_1+(i,0):i\in[0,t-1]\}$, $P^{**}_3=\{P_3+(i,0):i\in[p,p+t-1]\}$, $P^{**}_4=\{P_4+(i,0):i\in[t,p-1]\}$,
and $Q^{**}_1=\{Q_1+(i,2):i\in[0,p-1]\}$.
Now $P^{**}_1, P^{**}_3, P^{**}_4, Q^{**}_1$ is a partition of $Y_2$.

% Let $P^*_1=\{P_1+(i,0): i\in[0, t-1]\}$, $P^*_2=\{P_2+(i,0):i\in[p, p+t-1]\}$, $P^*_3=\{P_3+(i,0):i\in[p, p+t-1]\}$, and $P^*_4=\{P_4+(i, 0):i\in[t, p-1]\}$.

% $(i)$:
% Now $P^*_1$, $P^*_2$, $P^*_4$ is a partition of $Y_1$.

% $(ii)$:
% Obtain $Q_1$ by stretching $T_1$ from Lemma~\ref{lem:block2} in the $e_1$ direction by $q$; in other words $Q_1=\{(1, 1), (q+1, 1), (1, 2)\}$.
% Let $Q^*_1=\{Q_1+(i, 2): i\in[0, p-1]\}$.
% Now $P^*_1$, $P^*_3$, $P^*_4$, $Q^*_1$ is a partition of $Y_2$.

See Figure~\ref{fig:layer2}
for an illustration.
\end{proof}
}

% 그림 수정
 \begin{figure}[h]
	\begin{center}
  \includegraphics[scale=1]{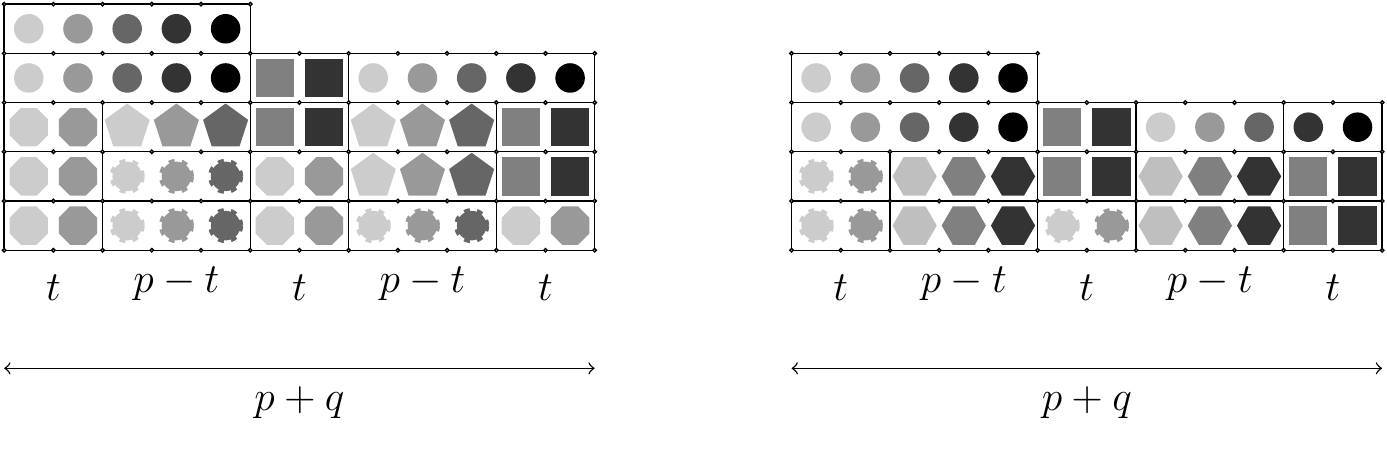}
  \caption{Figure for Lemma~\ref{lem:layer2}. Same shape and same shade means same block. Same shape and different shade means same type of block, but different block.}
  \label{fig:layer2}
	\end{center}
\end{figure}

\section{Main result}\label{sec:main}

A set $S\subset \ZZ^2$ is called a {\it layer}, and $S$ is an {\it $a$-nice layer} if $S$ is of the form $([a]\times[b])\cup\{(i,b+1):i\in[c]\}$ where $a, b, c$ are integers.
Given an element $(x, y)\in S$, the {\it row} of $(x, y)$ is the set of elements in $S$ with the same second coordinate.
Note that the sets $X_1$ and $X_2$ from Lemma~\ref{lem:layer1} are $q$-nice layers and the sets $Y_1$ and $Y_2$ from Lemma~\ref{lem:layer2} are $(p+q)$-nice layers.

For convenience, we will say a set with gap sequence $d_1, \ldots, d_{n-1}$ is a $(d_{\sigma(1)}, \ldots, d_{\sigma(n-1)})$-set for any permutation $\sigma$ of $[n-1]$.

\begin{lem}\label{lem:new}
For each $i\in[n]$, let $S_i$ be an $a$-nice layer that can be covered with height $h(S_i)$ by $V$-blocks, and let $l=\lcm{_{i\in[n]}\{h(S_i)\}}$.
For $r\geq 1-d+d\sum_{i\in[n]}|S_i|$ and positive integers $d, p, q$ with $q\geq p$, the set $\bigcup_{j\in[l]}(d[\sum_{i\in[n]}|S_i|]+(j-1)r)$ can be partitioned into
\begin{enumerate}[$(i)$]
\item $(dp, da, r)$-sets when $V=\{(pe_1, e_2, e_3)\}$.
\item $(dp, d(a-p), r)$-sets and $(dq,d(a-q),r)$-sets when $V=\{(pe_1, e_2-pe_1, e_3),(qe_1, e_2-qe_1, e_3)\}$.
\end{enumerate}
\end{lem}
\sout
{
\begin{proof}
Let $<$ be an ordering of the elements of $\{(S_i\times l, i): i\in[n]\}$ such that $((x_1, y_1, z_1), i_1)<((x_2, y_2, z_2), i_2)$
if $(a)$ $z_1<z_2$
or $(b)$ $z_1=z_2$ and $i_1<i_2$
or $(c)$ $z_1=z_2$, $i_1=i_2$, and $y_1<y_2$
or $(d)$ $z_1=z_2$, $i_1=i_2$, $y_1=y_2$, and $x_1<x_2$.
%This ordering of $[n]\times\ZZ^2$ under $<$ gives an ordering of elements of $\{h(S_i)\times S_i: i\in[n]\}$.
This ordering $<$ gives a natural bijection $\varphi$ between $\{(S_i\times l, i): i\in[n]\}$ and $\bigcup_{j\in[l]}(d[\sum_{i\in[n]}|S_i|]+(j-1)r)$.
Note that the condition $r\geq 1-d+d\sum_{i\in[n]}|S_i|$ is needed to ensure that $\varphi$ is a bijection.

$(i)$
Assume each $S_i$ can be covered by $(pe_1, e_2, e_3)$-blocks, and let $u$ and $v$ be two elements of one particular block.
If $|u-v|=pe_1$, then $|\varphi(u)-\varphi(v)|=dp$ since $u$ and $v$ are in the same row.
If $|u-v|=e_2$, then $|\varphi(u)-\varphi(v)|=da$ since the lower row of $u$ and $v$ has exactly $a$ elements.
If $|u-v|=e_3$, then $|\varphi(u)-\varphi(v)|=r$ since $u$ and $v$ must be in different layers.

Therefore, $\bigcup_{j\in[l]}(d[\sum_{i\in[n]}|S_i|]+(j-1)r)$ can be partitioned into $(dp, da, r)$-sets.

$(ii)$
Assume each $S_i$ can be covered by $\{(pe_1, e_2-pe_1, e_3),(qe_1,e_2-qe_1,e_3)\}$-blocks, and let $u$ and $v$ be two elements of a $(pe_1,e_2-pe_1,e_3)$-block.
If $|u-v|=pe_1$, then $|\varphi(u)-\varphi(v)|=dp$ since $u$ and $v$ are in the same row.
If $|u-v|=e_2-pe_1$, then $|\varphi(u)-\varphi(v)|=d(a-p)$ since the lower row of $u$ and $v$ has exactly $a$ elements.
If $|u-v|=e_3$, then $|\varphi(u)-\varphi(v)|=r$ since $u$ and $v$ must be in different layers.
The case when $u$ and $v$ are two elements of a $(qe_1,e_2-qe_1,e_3)$-block is analogous.

Therefore, $\bigcup_{j\in[l]}(d[\sum_{i\in[n]}|S_i|]+(j-1)r)$ can be partitioned into $(dp, d(a-p),r)$-sets and $(dq, d(a-q), r)\}$-sets.
\end{proof}
}

\begin{lem}\label{lem:new2}
%Let both $r_1$ and $r_2$ be linear combinations of positive integers $p$ and $q$, and let $d=gcd(p,q)$.
Given positive integers $r_1$ and $r_2$, let $d=\gcd(r_1, r_2)$, and also let $p$ and $q$ be positive integers such that $q\geq p$ and $p/d$ and $q/d$ are also integers.
For all integers $r\geq d(r_1/d-1)(r_2/d-1)$,
there is an interval of $\ZZ$ that can be partitioned into $(p, q, r)$-sets
if $L_1$ and $L_2$ is a layer of size $r_1/d$ and $r_2/d$, respectively, and both $L_1$ and $L_2$ are
%If the set of integers can be partitioned into both $(p/d,q/d,r_1/d)$-sets and $(p/d, q/d, r_2/d)$-sets,
\begin{enumerate}[$(i)$]
\item
 $(q/d)$-nice layers that can be covered by $(pe_1/d, e_2,e_3)$-blocks.

\item
$(p/d+q/d)$-nice layers that can be covered by $\{(pe_1/d, e_2-pe_1/d,e_3),(qe_1/d, e_2-qe_1/d, e_3)\}$-blocks.
\end{enumerate}
\end{lem}
%\sout
{
\begin{proof}
Let $l=\lcm\{h(L_1),h(L_2)\}$.
An integer $s$ is {\it good} if $(r_1/d-1)(r_2/d-1)\leq s\leq {\frac{r-1+d}{ d}}$.
Note that a good integer $s$ satisfies $r\geq 1-d+ds$.
% and fix $r$.
%Since $r_1/d$ and $r_2/d$ are coprime and $r\geq d(r_1/d-1)(r_2/d-1)\geq (r_1/d-1)
Since $r_1/d$ and $r_2/d$ are coprime, a good $s$ can be expressed as a linear combination of $r_1/d$ and $r_2/d$ with nonnegative coefficients.
%to obtain that %for any $s\in[(r_1/d-1)(r_2/d-1), r]$,
Therefore, given a good $s$, the set $T(s)=\bigcup_{j\in[l]}(d[s]+(j-1)r)$ can be partitioned into $(p, q, r)$-sets by Lemma~\ref{lem:new} since $(p, q, r)$-sets are also $(q, p, r)$-sets, for both $(i)$ and $(ii)$.

%For integers $j$ and $t$, let $T_j(t)=T(t)+j$ so that $T_j(t)$ is a copy of $T(t)$ shifted by $j$.
%Since $r\geq d(r_1/d-1)(r_2/d-1)$, we know that $\lfloor{r\over d}\rfloor\geq (r_1/d-1)(r_2/d-1)$.
Let $r'=r-d\lfloor{r/d}\rfloor$ so that $r'\in[0,d-1]$.
If $r'\neq 0$, then both $\lfloor{r\over d}\rfloor$ and $\lfloor{r\over d}\rfloor+1$ are good, and therefore by the above paragraph, both $T(\lfloor{r\over d}\rfloor)$ and $T(\lfloor{r\over d}\rfloor+1)$ can be partitioned into $(p, q, r)$-sets.
Now, $\bigcup_{i\in[r']}(T(\lfloor{r\over d}\rfloor+1)+i)\cup\bigcup_{i\in[r'+1,d-1]}(T(\lfloor{r\over d}\rfloor)+i)=[lr]+d$.

If $r'= 0$, then $\lfloor{r\over d}\rfloor$ is good, and therefore by the first paragraph, $T(\lfloor{r\over d}\rfloor)$ can be partitioned into $(p, q, r)$-sets.
Now, $\bigcup_{i\in[1,d-1]}(T(\lfloor{r\over d}\rfloor)+i)=[lr]+d$.

In both cases, $[lr]+d$ can be partitioned into $(p, q, r)$-sets.
\end{proof}
}

%\begin{lem}
%%Let both $r_1$ and $r_2$ be linear combinations of positive integers $p$ and $q$, and let $d=gcd(p,q)$.
%For positive integers $r_1$ and $r_2$, let $d=\gcd(r_1, r_2)$.
%If there are $(q/d, b, c)$-layers $L_1$ and $L_2$ of size $r_1/d$ and $r_2/d$, respectively, that can be covered by $(pe_1/d, e_2, e_3)$-blocks,
%%If the set of integers can be partitioned into both $(p/d,q/d,r_1/d)$-sets and $(p/d, q/d, r_2/d)$-sets,
%then there is an interval of $\ZZ$ that can be partitioned into $(p, q, r)$-sets for all integers $r\geq d(r_1/d-1)(r_2/d-1)$.
%\end{lem}
%\sout
%{
%\begin{proof}
%Note that $r_1/d$ and $r_2/d$ are coprime.
%Let $l=lcm\{h(L_1),h(L_2)\}$ and fix $r$.
%
%Since $r_1/d$ and $r_2/d$ are coprime and $r\geq d(r_1/d-1)(r_2/d-1)\geq (r_1/d-1)(r_2/d-1)$, we know that $r$ can be expressed as a linear combination of $r_1/d$ and $r_2/d$ with nonnegative coefficients.
%Since $L_1$ and $L_2$ is a $(q/d, b, c)$-layer of size $r_1/d$ and $r_2/d$, respectively, by Lemma~\ref{lem:new}, we know that for any $s\in[(r_1/d-1)(r_2/d-1), r]$, the set $T(s)=\bigcup_{j\in[l]}(d[s]+(j-1)r)$ can be partitioned into $(p, q, r)$-sets.
%
%Let $r'\in[0, d]$ so that $r=dr_0+r'$.
%For integers $j$ and $t$, let $T_j(t)=T(t)+j$ so that $T_j(t)$ is a copy of $T(t)$ shifted by $j$.
%Now, $\bigcup_{i\in[r']}T_i(s+1)\cup\bigcup_{i\in[r',d-1]}T_i(s)=[d,ds+d]$.
%Therefore, $[d, ds+d]$ can be partitioned into $(p, q, r)$-sets.
%\end{proof}
%}

\begin{thm}\label{thm:big}
For positive integers $p, q$ with $q\geq2p$, if $r\geq 4q(4q-1)$, then there is an interval of $\ZZ$ that can be partitioned into $4$-sets of the same gap sequence $p, q, r$.
\end{thm}
\begin{proof}
Since $q\geq 2p$, the layer $X_1$ and $X_2$ in Lemma~\ref{lem:layer1} is a $q$-nice layer of size $4q$ and $4q+1$, respectively, that can be covered by $(pe_1, e_2, e_3)$-blocks.
Note that $\gcd(4q, 4q+1)=1$.
Thus, by Lemma~\ref{lem:new2}, there is an interval of $\ZZ$ that can be partitioned into $(p, q, r)$-sets for all integers $r\geq 4q(4q-1)$.
\end{proof}

\begin{thm}\label{thm:small}
For positive integers $p, q$ with $q\in[p,2p]$, if $r\geq {1\over\gcd(p, q)}({5p+4q}-\gcd(p,q))({4p+3q}-\gcd(p,q))$, then there is an interval of $\ZZ$ that can be partitioned into $4$-sets of the same gap sequence $p, q, r$.
\end{thm}
\begin{proof}
Since $q\in[p, 2p]$, the layer $Y_1$ and $Y_2$ in Lemma~\ref{lem:layer2} is a $(p+q)$-nice layer of size $5p+4q$ and $4p+3q$, respectively, that can be covered by $\{(pe_1, e_2-pe_1, e_3), (qe_1, e_2-qe_1, e_3)\}$-blocks.
Note that $\gcd(5p+4q, 4p+3q)=\gcd(p,q)$.
Thus, by Lemma~\ref{lem:new2}, there is an interval of $\ZZ$ that can be partitioned into $(p, q, r)$-sets for all integers $r\geq\gcd(p, q)({5p+4q\over\gcd(p,q)}-1)({4p+3q\over\gcd(p,q)}-1)$.
\end{proof}

%\begin{thmmain}
%\mainstatement
%\end{thmmain}
%\begin{proof}
Theorem~\ref{thm:main} follows directly from Theorem~\ref{thm:big} and Theorem~\ref{thm:small}.
%\end{proof}

\section{Future directions and open questions}\label{sec:open}

As noted in the introduction, we omit some improvements on the constants of the threshold on $r$ in Theorem~\ref{thm:main}.
For example, it is not hard to show that $([ q]\times[4])\cup\{(q+j, 4):j\in[i]\}$ can be covered by $(pe_1,e_2,e_3)$-blocks for all $i\in[0,p]$, but we only provided the proof when $i\in\{0, 1\}$.
Finding more blocks in Lemma~\ref{lem:block1} and Lemma~\ref{lem:block2} will help finding more layers that can be covered in Lemma~\ref{lem:layer1} and Lemma~\ref{lem:layer2}, and appropriate combinations will improve the constants on the threshold on $r$.

\bigskip

We approached Question~\ref{ques} with the mind set of allowing all gap sequences, but focusing on the case when $n=4$, which is the first open case.
Another approach is to investigate the question for all $n$, but for special gap sequences.
The following conjecture was explicitly made in~\cite{2005Na}:

\begin{conj}[\cite{2005Na}]
There is an interval of $\ZZ$ that can be partitioned into $(k+l+1)$-sets with the same gap sequence $p_1, \ldots, p_k, q_1, \ldots, q_l$ where $p_1=\cdots=p_k$ and $q_1=\cdots=q_l$.
\end{conj}

The truth of Question~\ref{ques} when $n=3$ is equivalent to this conjecture when $k=l=1$.
Some partial results on this conjecture were made in~\cite{2005Na}.

\bigskip

As mentioned in the introduction, Gordon~\cite{1980Go} investigated the question in higher dimensions.
We iterate some open questions for the $2$-dimensional case.
As it is known that there is a $6$-set of $\ZZ^2$ that does not tile $\ZZ^2$ under the Euclidean group actions, the following statement is stated as ``conceivable'' in~\cite{1980Go}:

\begin{ques}[\cite{1980Go}]
Does every set $S$ of $\ZZ^2$ with $|S|\leq 5$ tile $\ZZ^2$ under the Euclidean group actions?
\end{ques}

Gordon~\cite{1980Go} also proved that a $3$-set of $\ZZ^2$ tiles $\ZZ^+\times\ZZ$ under the Euclidean group actions, whereas there is a $4$-set of $\ZZ^2$ that does not.
Actually, the same $4$-set does not even tile $\ZZ^+\times\ZZ^+$ under the Euclidean group actions, but Gordon~\cite{1980Go} proved that every $2$-set does tile $\ZZ^+\times\ZZ^+$ under the Euclidean group actions.
To the authors' knowledge, the following question, which appeared in~\cite{1980Go}, is still open:

\begin{ques}[\cite{1980Go}]
Does every $3$-set of $\ZZ^2$ tile $\ZZ^+\times\ZZ^+$ under the Euclidean group actions?
\end{ques}

\section*{Acknowledgments}

The authors thank Jae Baek Lee for introducing the problem to the authors.

\bibliographystyle{alpha}

%\bibliography{integer_partition}

%\begin{thebibliography}{99}
%
%\end{thebibliography}

\end{document}